\documentclass{amsart}
\usepackage{amsmath, amssymb, mathrsfs, verbatim, multirow}

\theoremstyle{definition}
\newtheorem{Teo}{Theorem}[section]
\newtheorem{Def}[Teo]{Definition}

\newtheorem{Obs}[Teo]{Remark}
\newtheorem{Lema}[Teo]{Lemma}
\newtheorem{Cor}[Teo]{Corollary}

\newcommand{\N}{\mathbb{N}}

\newcommand{\Llr}{\Longleftrightarrow}
\newcommand{\lra}{\longrightarrow}

\newcommand{\VR}{\mathcal{O}}
\newcommand{\PI}{\mathfrak{p}}
\newcommand{\MI}{\mathfrak{m}}

\begin{document}
\title{ Reduction of local uniformization to the rank one case}

\date{}

\author{Josnei Novacoski}
\address{Department of Mathematics and Statistics,
University of Saskatchewan,
Saskatoon, \newline \indent
SK S7N 5E6, Canada}
\email{jan328@mail.usask.ca}

\author{Mark Spivakovsky}


\begin{abstract}
The main result of this paper is that in order to prove the local uniformization theorem for local rings it is enough to prove it for rank one valuations. Our proof does not depend on the nature of the class of local rings for which we want to prove local uniformization. We prove also the reductions for different versions of the local uniformization theorem.
\end{abstract}
\maketitle

\section{Introduction}

Zariski's idea to solve the resolution of singularities problem for an algebraic variety was to prove local uniformization for all valuations and use the quasi-compactness of the Zariski space to glue the solutions together and construct a global resolution. In \cite{Zar_2} he achieved local uniformization for valuations of any field of characteristic zero. Then in \cite{Zar_3} he achieved the glueing of these solutions for particular cases. In 1964, Hironaka proved in his celebrated paper \cite{Hir_1}, that resolution of singularities can be obtained for algebraic varieties over a field of characteristic zero, without using the local approach via valuation theory.

Since then many results were obtained towards resolution of singularities and local uniformization in the positive characteristic case, but local uniformization in arbitrary dimension and characteristic remains an open problem. Abhyankar proved in \cite{Ab_1} that local uniformization can be done for algebraic surfaces in any characteristic and used this fact to prove resolution of singularities for surfaces (see \cite{Ab_2} and \cite{Ab_3}). He also proved local uniformization (and resolution of singularities) for threefolds over fields of characteristic other than 2, 3 and 5 \cite{Ab_4}. Cossart and Piltant proved in \cite{Cos_2} resolution of singularities (and in particular, local uniformization) for algebraic varieties of dimension three. Another important case is the one proved by Knaf and F.-V. Kuhlmann in \cite{Kna}, namely that all Abhyankar valuations admit local uniformization.

Another approach to local uniformization is by alteration of the algebraic function field of the given variety. In this sense, De Jong proved in \cite{Dej} that resolution of singularities can be obtained after a finite extension of the function field. Knaf and F.-V. Kuhlmann refined in \cite{Kna_1} this result for local uniformization, proving that this extension can be taken as being separable or a Galois extension. Also, Temkin proved in \cite{Tem} that this extension can be chosen as being purely inseparable.

To prove local uniformization it is convenient to work with rank one valuations. This is because, for instance, complete valued fields of rank one are henselian but this is not true, in general, for higher rank valuations. A natural way to handle local uniformization is to reduce the problem to rank one valuations. In this paper, we prove this reduction. Our main results are Theorems \ref{Teo_1}, \ref{Teo_2} and \ref{Teo_3} below.

Roughly speaking, the local uniformization problem asks whether for every valuation $\nu:K^\times\lra \Gamma$ centered at $\left(R,\MI\right)$ we can find another local ring $\left(R^{(1)},\MI^{(1)}\right)$, birationally dominating $\left(R,\MI\right)$, with $R^{(1)}\subseteq \VR_\nu\subseteq K=Quot(R)$ such that $R^{(1)}$ is regular. Precise definitions of the three different types of local uniformization that we consider are given in \S2.3 below.
Let $\mathcal{N}$ be the category of all Noetherian local rings and let $\mathcal{M}\subseteq\mathcal{N}$ be a subcategory of $\mathcal{N}$ which is closed under taking homomorphic images, finitely generated birational extensions and localizations. Grothendieck conjectured that the subcategory which optimizes local uniformization and resolution of singularities is the category of all quasi-excellent local rings. Our proofs do not depend of the nature of the category $\mathcal{M}$.

\begin{Teo}\label{Teo_1}
Assume that for every Noetherian local ring $(R,\MI)$ in $\mathcal{M}$, every rank one valuation centered in $(R,\MI)$ admits local uniformization. Then all the valuations centered in members of $\mathcal{M}$ admit local uniformization.
\end{Teo}

A stronger version of local uniformization, called weak embedded local uniformization problem, is whether for every given finite subset $Z$ of $R$ we can find a regular local ring $(R^{(1)},\MI^{(1)})$, birationally dominating $(R,\MI)$ and dominated by $\VR_\nu$, and a regular system of parameters $u=(u_1,\ldots,u_d)$ of $R^{(1)}$ such that all elements of $Z$ are monomials in $u$.

\begin{Teo}\label{Teo_2}
Assume that for every Noetherian local ring $(R,\MI)$ in $\mathcal{M}$, every rank one valuation centered in $(R,\MI)$
 admits weak embedded local uniformization. Then all the valuations centered in members of $\mathcal{M}$ admit weak embedded local uniformization.
\end{Teo}

We order the elements of the set $Z$ above by their values, i.e., $Z=\{f_1,\ldots,f_q\}$ such that $\nu(f_1)\leq\ldots\leq\nu(f_q)$. Another version of the local uniformization problem is whether we can find a regular local ring with regular system of parameters $u$ as before such that the elements $f_i$ are monomials in $u$ and moreover $f_1\mid_{R^{(1)}}\ldots\mid_{R^{(1)}} f_q$. This version is called embedded local uniformization.

\begin{Teo}\label{Teo_3}
Assume that for every Noetherian local ring $(R,\MI)$ in $\mathcal{M}$, every rank one valuation centered in $(R,\MI)$
 admits embedded local lniformization. Then all the valuations centered in members of $\mathcal{M}$ admit embedded local uniformization.
\end{Teo}

\section{Preliminaries}

We will assume that the reader is familiar with basic facts about valuations and will use them without further reference. For general valuation theory we recommend \cite{FvkB}, \cite{Pre} and \cite{Zar_6} and for basic commutative algebra we suggest \cite{Ath}.

\subsection{Decomposition of a valuation}

Let $\nu:K^\times\lra \Gamma$ be a valuation of the field $K$. We denote by $\VR_\nu$ its valuation ring and by $\MI_\nu$ its maximal ideal. We define the value group $\nu K$ of $\nu$ as the subgroup of $\Gamma$ generated by $\{\nu a\mid a\in K^\times\}$, the residue field associated to $\nu$ as the field $K_\nu:=\VR_\nu/\MI_\nu$ and the place $P_\nu$ associated with $\nu$ as
\begin{displaymath}
\begin{array}{rccl}
P_\nu: & K &\lra    &K_\nu \cup\{\infty\}\\
          & a &\longmapsto &\left\{\begin{array}{ll} a+\MI_\nu &\textnormal{ if }a\in\VR_\nu \\ \infty &\textnormal{ otherwise} \end{array}
\right. .
\end{array}
\end{displaymath}
Also, given a place $P:K\lra k\cup\{\infty\}$ we can define the valuation
\begin{displaymath}
\begin{array}{rccl}
\nu_P: & K^\times &\lra    &K^\times/\VR_P^\times\\
          & a &\longmapsto &a\VR_P^\times
\end{array}
\end{displaymath}
associated to $P$ where $\VR_P:=P^{-1}(k)$ is the valuation ring of $P$.

Take $\Delta$ a convex subgroup of $\Gamma$ and let
\[
\pi_\Delta: \Gamma\lra \Gamma/\Delta
\]
be the canonical epimorphism of $\Gamma$ onto the quotient group $\Gamma/\Delta$. Then the function
$$
\nu_\Delta=\pi_\Delta\circ\nu
$$
is a valuation of $K$ whose valuation ring $\VR_{\nu_\Delta}$ contains $\VR_\nu$.
The valuation $\nu$ induces a valuation $\overline{\nu}_\Delta:K_{\nu_\Delta}^\times\lra \Delta$ by setting
\[
\overline{\nu}_\Delta(a+\MI_{\nu_\Delta})=\nu(a)
\]
for all $a\in \VR_\nu\setminus\MI_\nu$.

Consider the places $P_\Delta$ and $\overline{P}_\Delta$ associated to $\nu_\Delta$ and $\overline{\nu}_\Delta$, respectively, and let $P_\Delta\circ\overline{P}_\Delta$ be the composition of $P_\Delta$ and $\overline{P}_\Delta$ as functions, i.e.
\begin{displaymath}
\begin{array}{rccl}
P_\Delta\circ\overline{P}_\Delta: & K &\lra    &K_{\overline{\nu}} \cup\{\infty\}\\
          & a &\longmapsto &\left\{\begin{array}{ll} \left(aP_\Delta\right)\overline{P}_\Delta &\textnormal{ if }a\in\VR_{\nu_\Delta} \\ \infty &\textnormal{ otherwise} \end{array}
\right. .
\end{array}
\end{displaymath}

Then $P_\Delta\circ\overline{P}_\Delta$ is a place of $K$ and we can consider the valuation $\nu_\Delta\circ\overline{\nu}_\Delta$ in $K$ as the valuation associated to $P_\Delta\circ\overline{P}_\Delta$. It is easy to see that the valuation rings of $\nu$ and $\nu_\Delta\circ\overline{\nu}_\Delta$ are the same, and consequently these valuations are equivalent. Since we are considering classes of valuations we will say that $\nu=\nu_\Delta\circ\overline{\nu}_\Delta$ and we will say that this is \textbf{the decomposition of $\nu$ associated to $\Delta$}.

\begin{Def}
Let $(R,\MI)$ be a Noetherian local ring and $K=Quot(R)$. Given a valuation $\nu:K^\times\lra \Gamma$ we say that $\nu$ is \textbf{centered} at $\left(R,\MI\right)$ if $R\subseteq \VR_\nu$ and $\mathfrak{m}=R\cap\mathfrak{m}_\nu$. If
$$
R\subseteq\VR_\nu\subseteq K=Quot(R)
$$
is a ring (not necessarily local) we say that $\mathfrak{p}=R\cap\mathfrak{m}_\nu$ is \textbf{the center} of $\nu$ in $R$.
\end{Def}

Let $\nu$ be a valuation of $K=Quot(R)$ centered at the local ring $(R,\MI)$. If $\nu=\nu_1\circ\nu_2$ then $\nu_1$ has a center $\PI=\MI_{\nu_1}\cap R$ which is a subset of $\MI$. Define now the subgroup $\Delta\subseteq \nu K$ as the convex hull of
\[
\{\nu(x)\mid \nu_1(x)=0\}
\]
in $\nu K$. Then the valuations $\nu_1$ and $\nu_2$ are equivalent to $\nu_\Delta$ and $\overline{\nu}_\Delta$ respectively. In particular, for any element $x\in \MI\setminus\PI$ we have $\nu(y)>\nu(x)$ for all $y\in \PI$.

\begin{Def}
Let $R$ be a ring and let $\PI$ be a prime ideal of $R$. Then the field
\[
\kappa(\PI):=R_\PI/\PI R_\PI
\]
is called \textbf{the residue field} of $\PI$.
\end{Def}

\begin{Lema}\label{Le_2}
Let $\nu=\nu_1\circ\nu_2$ be a valuation of $K=Quot(R)$ centered at the local domain $(R,\MI)$. If $\PI\subseteq\MI$ is the center of $\nu_1$ in $R$ then $\kappa(\PI)$ embeds naturally in $K_{\nu_1}$ and the restriction of $\nu_2$ to $\kappa(\PI)$ is centered at $(R/\PI,\MI/\PI)$.
\end{Lema}
\begin{proof}
It is easy to see that $R_\PI\subseteq \VR_{\nu_1}$ and that $\PI R_\PI=R_\PI\cap\MI_{\nu_1}$. Therefore,
\[
\kappa(\PI)=R_\PI/\PI R_\PI\hookrightarrow \VR_{\nu_1}/\MI_{\nu_1}=K_{\nu_1}.
\]

On the other hand, $\kappa(\PI)=Quot(R/\PI)$ so it remains to show that $\MI_{\nu_2}\cap R/\PI=\MI/\PI$. We will prove this equality by proving that an element $a\in R$ belongs to its left hand side if and only if it belongs to the right hand side. Take an element $a\in R$. If $a\in \PI$ then $a+\PI=0+\PI\in\MI/\PI\cap \left(\MI_{\nu_2}\cap R/\PI\right)$, so we assume that $a\notin \PI$. In this case, $\nu_2(a+\PI)=\nu(a)$ and we have
\[
a+\PI\in \MI/\PI\Llr a\in\MI\Llr\nu(a)>0\Llr \nu_2(a+\PI)>0\Llr a+\PI\in\MI_{\nu_2}.
\]
\end{proof}

\begin{Obs}\label{RnusurjectstoRnu2} Let $\PI_\Delta=\MI_{\nu_1}\cap\VR_\nu$. Applying the above Lemma to the local ring $R=\VR_\nu$, we see that there is a natural surjective homomorphism $\Phi_\Delta:\VR_\nu\rightarrow\VR_{\nu_2}$, whose kernel is $\PI_\Delta$.
\end{Obs}

\begin{Lema}\label{localization}
\begin{description}
\item[(1)] Let
\begin{equation}
R\hookrightarrow R'\label{eq:RinR'}
\end{equation}
be an injective homomorphism of domains, having a common field of fractions $K$. Let $\nu$ be a valuation of $K$ such that $R'\subset\VR_\nu$. Let $\MI$ denote the center of $\nu$ in $R$, and $\MI'$ the center of $\nu$ in $R'$. Then
\begin{equation}
R_{\MI}\subset R'_{\MI'}\label{eq:RinR'localized}
\end{equation}
(viewed as subrings of $K$).
\medskip

\item[(2)] Let $S\subset R$ be a multiplicative subset, such that
\begin{equation}
S\cap\MI=\emptyset.\label{eq:Snotincenter}
\end{equation}
Assume that $R'=R_S$. Then the inclusion (\ref{eq:RinR'localized}) is, in fact, an equality.
\end{description}
\end{Lema}
\begin{proof}
\begin{description}
\item[(1)] The inclusion (\ref{eq:RinR'}) induces the inclusion
\begin{equation}
R\setminus\MI\subset R'\setminus\MI'.\label{eq:RinR'complement}
\end{equation}
The desired inclusion (\ref{eq:RinR'localized}) follows immediately from (\ref{eq:RinR'}) and (\ref{eq:RinR'complement}).
\medskip

\item[(2)] The assumption (\ref{eq:Snotincenter}) implies that $R'=R_S\subset R_{\MI}$. Then
\begin{equation}
R'_{\MI'}\subset R_{\MI}\label{eq:R'inRlocalized}
\end{equation}
by the first part of the Lemma, and the result follows.
\end{description}
\end{proof}

\subsection{Regularity of a ring}

Given a local Noetherian ring $\left(R,\mathfrak{m},k\right)$ (not necessarily equi-characteristic) we define the dimension of $R$, denoted by $\dim R$, as the Krull dimension of $R$, i.e., the maximum length of chains of prime ideals in $R$. For convenience, we will sometimes write only $R$ or $(R,\MI)$ for $\left(R,\mathfrak{m},k\right)$.

\begin{Obs}
[Theorem 11.14 of \cite{Ath}] The dimension of $\left(R,\mathfrak{m}\right)$ is equal to the minimal number of generators of an $\mathfrak{m}$-primary ideal.
\end{Obs}

Let $(R,\MI)$ be a Noetherian local ring and let $u_1,\ldots,u_d\in\MI$. We say that
$$
u=(u_1,\ldots,u_d)
$$
is \textbf{a system of parameters} of $R$ (or of $\MI$) if $u_1,\ldots,u_d$ generate an $\MI$-primary ideal. An element $f\in R$ is said to be a monomial in $u$ if there exists
$$
\gamma:=\left(\gamma^{(1)},\ldots,\gamma^{(d)}\right)\in \left(\N\cup\{0\}\right)^d
$$
and $c\in R^\times$ such that
\[
f=cu^\gamma:=c\prod_{i=1}^du_i^{\gamma^{(i)}}.
\]
In this case we define $|\gamma|:=\displaystyle\sum_{i=1}^d \gamma^{(i)}$.

\begin{Obs}
We will say sometimes ``the local ring $(R,u)$" meaning that $R$ is a local ring with maximal ideal $\MI$ such that $u=(u_1,\ldots,u_d)$ is a fixed set of generators of $\MI$.
\end{Obs}

The local ring $R$ is said to be regular if $\mathfrak{m}$ can be generated by $\dim R$ many elements. In this case, a set of generators $\left(u_1,\ldots,u_d\right)$ of $\MI$ such that $d=\dim R$ is called a regular system of parameters of $R$.

\begin{Def}
Let $\left(R,\MI\right)$ be a Noetherian local domain with quotient field $K$ and $\nu$ be a valuation of $K$ centered at $R$. A \textbf{local blowing up} of $\left(R,\MI\right)$ with respect to $\nu$ a local ring homomorphism
\[
\pi:\left(R ,\MI\right)\lra \left(R^{(1)},\MI^{(1)}\right)
\]
of the following form: take elements $a_i,b_i\in R$, $i=1,\ldots,r$ such that $\nu\left(b_i\right)\leq \nu\left(a_i\right)$ for all $i=1,\ldots, r$ and let
\[
R'=R\left[\frac{a_1}{b_1},\ldots,\frac{a_r}{b_r}\right]\textit{ and } \MI'=\MI_\nu\cap R'.
\]
The local ring $\left(R^{(1)},\MI^{(1)}\right)$ is the localization of $R'$ with respect to the prime ideal $\MI'$, that is,
\[
R^{(1)}=R'_{\MI'}=\left\{\left.\frac{x}{y}\in K\ \right|\ y\notin \MI'\right\}\textit{ and }\MI^{(1)}=\MI'R'_{\MI'}.
\]
The local blowing up now is the natural inclusion $\pi:R\lra R^{(1)}$. We will say that the local blowing up $\pi$ is \textbf{simple} if $r=1$, i.e., $R^{(1)}=R\left[\displaystyle\frac{a}{b}\right]_{\MI'}$.
\end{Def}

\begin{Lema}\label{Le_3}
Every local blowing up can be decomposed as a finite sequence of simple local blowing ups, i.e., given a local blowing up
\[
\pi:R\lra R\left[\frac{a_1}{b_1},\ldots,\frac{a_r}{b_r}\right]_{\MI'}
\]
we can find a finite sequence of simple local blowing ups
\[
(R,\MI)\lra \left(R^{(1)},\MI^{(1)}\right)\lra \cdots \lra \left(R^{(r)},\MI^{(r)}\right)
\]
such that $R^{(r)}=R\left[\displaystyle\frac{a_1}{b_1},\ldots,\frac{a_r}{b_r}\right]_{\MI'}$ and $\pi$ is the composition of the simple local blowing ups $\pi_i:R^{(i-1)}\lra R^{(i)}$ (where we set $R^{(0)}:=R$).
\end{Lema}

\begin{proof}
Define the rings
\[
R'^{(k)}=R\left[\frac{a_1}{b_1},\ldots,\frac{a_k}{b_k}\right]_{R\left[\frac{a_1}{b_1},\ldots,\frac{a_k}{b_k}\right]\cap \MI_\nu},\qquad 1\leq k\leq r,
\]
and let us define $R^{(k)}$ inductively by setting $R^{(0)}=R$ and
\[
R^{(k)}=R^{(k-1)}\left[\frac{a_k}{b_k}\right]_{R^{(k-1)}\left[\frac{a_k}{b_k}\right]\cap\MI_\nu},\qquad 1\leq k\leq r.
\]

The inclusions $\pi_k:R^{(k-1)}\lra R^{(k)}$ are all simple local blowing ups, so we just have to prove that $R^{(r)}=R'^{(r)}$ and we will have $\pi=\pi_r\circ\ldots\circ\pi_1$ because all the $\pi_k$ are inclusions. We will prove by induction that $R^{(k)}=R'^{(k)}$ for all $k=1,\ldots,r$ and we will be done.

By definition, $R^{(1)}=R'^{(1)}$ so assume that $k>1$ and that
\begin{equation}
R^{(k-1)}=R'^{(k-1)}.\label{eq:Rk-1inR'k-1}
\end{equation}
 Let us prove that $R^{(k)}=R'^{(k)}$. The inclusion $R'^{(k)}\subseteq R^{(k)}$ is trivial so it remains to prove that
\begin{equation}
R^{(k)}\subseteq R'^{(k)}.\label{eq:RkinR'k}
\end{equation}
To prove (\ref{eq:RkinR'k}), first note that $R'^{(k-1)}\subset R'^{(k)}$, hence $R^{(k-1)}\subset R'^{(k)}$ by (\ref{eq:Rk-1inR'k-1}). We have $\displaystyle\frac{a_k}{b_k}\in R'^{(k)}$ by definition, so
\begin{equation}
R^{(k-1)}\left[\frac{a_k}{b_k}\right]\subset R'^{(k)}.\label{eq:unlocalized}
\end{equation}
Now (\ref{eq:RkinR'k}) is given by Lemma \ref{localization} (1). This completes the proof of the Lemma.
\end{proof}

\begin{Lema}\label{Le_1}
Let $(R,\MI)$ be a Noetherian local domain and $\nu$ a valuation on $K=Quot(R)$ which is centered at $(R,\MI)$. Take an ideal $\mathcal{I}$ of $R$ and elements $u_1,\ldots,u_d \in \mathcal{I}$ which generate $\mathcal{I}$. Assume that $\nu(u_{i_1})=\nu(u_{i_2})\leq \nu(u_i)$ for $1\leq i\leq d$ and let
\[
R'=R\left[\frac{u_1}{u_{i_1}},\ldots,\frac{u_d}{u_{i_1}}\right]\textit{ and }R''=R\left[\frac{u_1}{u_{i_2}},\ldots,\frac{u_d}{u_{i_2}}\right].
\]
Also, consider the prime ideals $\MI'=R'\cap\MI_\nu\subseteq R'$ and $\MI''=R''\cap\MI_\nu\subseteq R''$. Then
\[
R'_{\MI'}=R''_{\MI''}.
\]
\end{Lema}

\begin{proof}
We may assume, without loss of generality, that $i_1=1$ and $i_2=2$. Since
\[
\nu\left(\displaystyle\frac{u_2}{u_1}\right)=\nu(u_1)-\nu(u_2)=0
\]
we have $\displaystyle\frac{u_2}{u_1}\in R''_{\MI''}$. Hence for every $i\in\{1,\dots,d\}$ we have $\displaystyle\frac{u_i}{u_1}=\frac{u_i}{u_2}\frac{u_2}{u_1}\in R''_{\MI''}$. Then
\begin{equation}
R'\subset R''_{\MI''}.\label{eq:R'inR''}
\end{equation}
 From (\ref{eq:R'inR''}) and Lemma \ref{localization} (1) we obtain
\[
R'_{\MI'}\subset R''_{\MI''}.
\]
The opposite inclusion is analogous.
\end{proof}

Take now an ideal $\mathcal{I}$ of $R$ and let $u_0\in \mathcal{I}$ be an element such that $\nu(u_0)\leq \nu(\alpha)$ for all $\alpha\in \mathcal{I}$. Complete $u_0$ to sets $\{u_0,u_1,\ldots, u_q\}$ and $\{u_0,u'_1,\ldots,u'_{q'}\}$ of generators of $\mathcal{I}$. It is easy to see that
\[
R\left[\frac{u_1}{u_0},\ldots,\frac{u_q}{u_0}\right]=R\left[\frac{u'_1}{u_0},\ldots,\frac{u'_{q'}}{u_0}\right]=:R'.
\]
This, together with Lemma \ref{Le_1} above, guarantees that given an ideal $\mathcal{I}$ of $R$ the local blowing up
\[
R\lra R'_{R'\cap\MI_\nu}
\]
is uniquely determined by $\mathcal{I}$ and is independent of the particular set of generators $\{u_0,u_1,\ldots, u_q\}$.

\begin{Def}
The local blowing up described above is said to be the local blowing up of $(R,\MI)$ \textbf{with respect to} $\nu$ \textbf{along} $\mathcal{I}$.
\end{Def}

We will now prove a few Lemmas which will be essential in the proofs of our main results. From here until the end of this section we will assume that $(R,\MI)$ is a Noetherian local ring and $\nu$ a valuation of $K=Quot(R)$ centered at $(R,\MI)$. Also, assume that $\nu$ can be decomposed as $\nu=\nu_1\circ\nu_2$ and write $\PI=R\cap\MI_{\nu_1}$ for the center of $\nu_1$ in $R$.

\begin{Lema}\label{Le_4}
Let
\[
\widetilde{\pi}:R_\PI\lra \widetilde{R}=R_\PI\left[\frac{a}{b}\right]_{\widetilde{\PI}}
\]
be a simple local blowing up with respect to $\nu_1$, where $\widetilde{\PI}=\MI_{\nu_1}\cap R_\PI\left[\displaystyle\frac{a}{b}\right]$, and assume that $\nu(a)\geq\nu(b)$. Consider the sequence of local blowing ups
\[
R\lra R^{(1)}=R\left[a,b\right]_{\MI'}\lra R^{(2)}=R^{(1)}\left[\frac{a}{b}\right]_{\MI''}
\]
with respect to $\nu$, where $\MI'=\MI_\nu\cap R\left[a,b\right]$ and $\MI''=\MI_\nu\cap R^{(1)}\left[\displaystyle\frac{a}{b}\right]$. If $\PI^{(2)}$ is the center of $\nu_1$ in $R^{(2)}$ then $R^{(2)}_{\PI^{(2)}}=\widetilde{R}$.
\end{Lema}
\begin{proof} Let $\PI^{(1)}$, $\PI_0^{(1)}$ denote the centers of $\nu_1$ in $R^{(1)}$ and in $R[a,b]$, respectively. We have $a,b\in R_{\PI}$ by definition, so $R\left[a,b\right]\subset R_{\PI}$ and
\begin{equation}
R[a,b]_{\PI_0^{(1)}}\subset R_{\PI}\label{eq:RabinRlocalized}
\end{equation}
by Lemma \ref{localization} (1). Since $\PI_0^{(1)}\subset\MI'$, we have
\begin{equation}
R^{(1)}_{\PI^{(1)}}=R[a,b]_{\PI_0^{(1)}}\label{eq:R1inRab}
\end{equation}
by Lemma \ref{localization} (2). Combining (\ref{eq:RabinRlocalized}) and (\ref{eq:R1inRab}), we obtain
\begin{equation}
R^{(1)}_{\PI^{(1)}}\subset R_{\PI}.\label{eq:R1inRlocalized}
\end{equation}
Now, from the natural inclusion $R\subset R^{(2)}$ we have $R_{\PI}\subset R^{(2)}_{\PI^{(2)}}$ by Lemma \ref{localization} (1). Since $\displaystyle\frac ab\in R^{(2)}_{\PI^{(2)}}$ by definition, we obtain $R_{\PI}\left[\displaystyle\frac ab\right]\subset R^{(2)}_{\PI^{(2)}}$, so
\begin{equation}
\widetilde{R}\subset R^{(2)}_{\PI^{(2)}}\label{eq:RtildeinR2}
\end{equation}
by Lemma \ref{localization} (1). For the opposite inclusion, Let $\PI_0^{(2)}$ denote the center of $\nu_1$ in $R^{(1)}\left[\displaystyle\frac ab\right]$. Since $\PI_0^{(2)}\subset\MI''$, we have
\begin{equation}
R^{(1)}\left[\displaystyle\frac ab\right]_{\PI_0^{(2)}}=R^{(2)}_{\PI^{(2)}}\label{eq:delocalization}
\end{equation}
by Lemma \ref{localization} (2). We have $R^{(1)}\subset R_{\PI}\subset\widetilde{R}$ by (\ref{eq:R1inRlocalized}) and $\displaystyle\frac ab\in\widetilde{R}$ by definition, so
\[
R^{(1)}\left[\displaystyle\frac ab\right]\subset\widetilde{R}.
\]
Hence
\begin{equation}
R^{(1)}\left[\frac ab\right]_{\PI_0^{(2)}}\subset\widetilde{R}\label{eq:inRtilde}
\end{equation}
by Lemma \ref{localization} (1). Combining (\ref{eq:inRtilde}) with (\ref{eq:delocalization}), we obtain
\begin{equation}
R^{(2)}_{\PI^{(2)}}\subset\widetilde{R}.\label{eq:R2inRtilde}
\end{equation}
This completes the proof.
\end{proof}

If in the Lemma above we had $\nu(a)< \nu(b)$, then we would have $\nu_1(a)=\nu_1(b)$. Indeed, since $\nu(a)< \nu(b)$ we get
\[
\frac{b}{a}\in \MI_\nu\subseteq\VR_\nu\subseteq\VR_{\nu_1}
\]
which guarantees that $\nu_1(b)\leq\nu_1(a)$. By definition of a local blowing up we have that $\nu_1(a)\leq\nu_1(b)$, so $\nu_1(a)=\nu_1(b)$. From Lemma \ref{Le_1} we conclude that $\widetilde{R}=R\left[\displaystyle\frac{b}{a}\right]_{\widetilde{\PI'}}$ where $\widetilde{\PI'}=\MI_{\nu_1}\cap R\left[\displaystyle\frac{b}{a}\right]$. Consider now the sequence of local blowing ups
\[
R\lra R^{(1)}=R[a,b]_{\MI'}\lra R^{(1)}\left[\frac{b}{a}\right]_{\MI''}
\]
with respect to $\nu$, where $\MI'=\MI_\nu\cap R[a,b]$ and $\MI''=\MI_\nu\cap R^{(1)}\left[\displaystyle\frac{b}{a}\right]$ and let $\PI^{(2)}=\MI_{\nu_1}\cap R^{(2)}$. By the previous Lemma we conclude again that $\widetilde{R}=R^{(2)}_{\PI^{(2)}}$. We have then proved the following Corollary.

\begin{Cor}\label{Cor_2}
For every simple local blowing up of $R_\PI$
\[
R_\PI\lra \widetilde{R}
\]
with respect to $\nu_1$ there exists a sequence of local blowing ups of $R$
\[
R\lra R^{(1)}\lra R^{(2)}
\]
with respect to $\nu$ such that $\widetilde{R}=R^{(2)}_{\PI^{(2)}}$, where $\PI^{(2)}$ is the center of $\nu_1$ in $R^{(2)}$.
\end{Cor}

\begin{Cor}
Let
\begin{equation}\label{eq_3}
R_\PI\lra \widetilde{R}^{(1)}\lra\cdots \lra \widetilde{R}^{(r)}
\end{equation}
be a sequence of local blowing ups with respect to $\nu_1$. Then there exists a sequence of local blowing ups
\[
R\lra R^{(1)}\lra\cdots\lra R^{(n)}
\]
with respect to $\nu$ such that $R^{(n)}_{\PI^{(n)}}=\widetilde{R}^{(r)}$, where $\PI^{(n)}$ is the center $\MI_{\nu_1}\cap R^{(n)}$ of $\nu_1$ in $R^{(n)}$. In particular, if $\widetilde{R}^{(r)}$ is regular then $R^{(n)}_{\PI^{(n)}}$ is regular.
\end{Cor}
\begin{proof}
From Lemma \ref{Le_3} we may assume that every local blowing up in the sequence (\ref{eq_3}) is simple. Applying Corollary \ref{Cor_2} to each of these simple local blowing ups and using induction on $r$, we get the desired sequence.
\end{proof}

We will now prove some facts about $\nu_2$. Let
\[
\phi:R/\PI\lra\overline{R}
\]
be an isomorphism of local rings and denote $\phi(a+\PI)$ by $\overline{a}$. Let $\overline{\nu}_2=\nu_2\circ\phi^{-1}$ and take elements $a,b\in R\setminus\PI$ such that $\nu(a)=\overline{\nu}_2(\overline{a})\geq \overline{\nu}_2(\overline{b})=\nu(b)$. Consider the rings
\[
R'=R\left[\frac{a}{b}\right]\textnormal{ and }\overline{R}'=\overline{R}\left[\frac{\overline{a}}{\overline{b}}\right]
\]
and the ideals $\MI'=R'\cap \MI_\nu$ and $\overline{\MI}'=\overline{R}'\cap\MI_{\overline{\nu}_2}$. Let
\[
R^{(1)}=R'_{\MI'}\textnormal{ and }\overline{R}^{(1)}=\overline{R}'_{\overline{\MI}'}
\]
and $\PI^{(1)}=R'_{\MI'} \cap \MI_{\nu_1}$.

\begin{Lema}\label{Le_7}
In the above situation we have $\overline{R}^{(1)}\cong R^{(1)}/\PI^{(1)}$ and $R_\PI= R^{(1)}_{\PI^{(1)}}$.
\end{Lema}

\begin{proof}
To prove that $R_\PI= R^{(1)}_{\PI^{(1)}}$ it is enough to prove that
\begin{equation}
R_\PI\supseteq R^{(1)}_{\PI^{(1)}}\label{eq:RPinR1P1}
\end{equation}
 because $R_\PI\subseteq R^{(1)}_{\PI^{(1)}}$ is trivial. To prove (\ref{eq:RPinR1P1}), note that $b\notin\PI$ by definition, hence $\displaystyle\frac ab\in R_\PI$, so $R\left[\displaystyle\frac ab\right]\subset R_\PI$. Now the inclusion (\ref{eq:RPinR1P1}) follows from Lemma \ref{localization} (1).
 
 To prove the first statement of the Lemma, first note that we have a natural surjective homomorphism $R\rightarrow\overline{R}$. We extend it to a surjective homomorphism $\Phi:R'\rightarrow\overline{R}'$ by sending $\displaystyle\frac{a}{b}$ to $\displaystyle\frac{\overline{a}}{\overline{b}}$. We have $\Phi(\MI')=\overline{\MI}'$ and $\Phi(R'\setminus\MI')=\overline{R}'\setminus\overline{\MI}'$, hence $\Phi$ extends to a surjective homomorphism $R^{(1)}\lra \overline{R}^{(1)}$ of localizations. By abuse of notation, we denote this new homomorphism also by $\Phi$.

It remains to show that $\ker (\Phi)=\PI^{(1)}$. By definitions, we have injective local homomorphisms $R^{(1)}\hookrightarrow\VR_\nu$ and $\overline{R}^{(1)}\hookrightarrow\VR_{\nu_2}$, and the homomorphism $\Phi$ is nothing but the restriction to $R^{(1)}$ of the homomorphism $\Phi_\Delta$ of Remark \ref{RnusurjectstoRnu2}. Hence
$$
\ker(\Phi)=\ker(\Phi_\Delta)\cap R^{(1)}=\PI_\Delta\cap R^{(1)}=\left(\MI_{\nu_1}\cap\VR_\nu\right)\cap R^{(1)}=\MI_{\nu_1}\cap R^{(1)}=\PI^{(1)},
$$
as desired.
\end{proof}

\begin{Def}
The simple local blowing up
\[
\pi:R\lra R\left[\frac{a}{b}\right]_{\MI'}
\]
constructed in the Lemma above is called the \textbf{lifting} of the simple local blowing up $\overline{\pi}$ from $R/\PI$ to $R$.
\end{Def}

\begin{Cor}\label{Cor_1}
Take a sequence of local blowing ups
\[
\left(R/\PI,\MI/\PI\right)\lra\left(\overline{R}^{(1)},\overline{\MI}^{(1)}\right)\lra\cdots \lra\left(\overline{R}^{(r)},\overline{\MI}^{(r)}\right),
\]
with respect to $\nu_2$. Then there exists a sequence of local blowing ups
\[
\left(R,\MI\right)\lra\left(R^{(1)},\MI^{(1)}\right)\lra\cdots\lra\left(R^{(n)},\MI^{(n)}\right),
\]
with respect to $\nu$ such that $R^{(n)}_{\PI^{(n)}}=R_\PI$ and $R^{(n)}/{\PI^{(n)}}\cong \overline{R}^{(r)}$, where $\PI^{(n)}=R^{(n)}\cap\MI_{\nu_1}$. In particular, if $R_\PI$ and $\overline{R}^{(r)}$ are regular, then so are $R^{(n)}_{\PI^{(n)}}$ and $R^{(n)}/{\PI^{(n)}}$.
\end{Cor}

\begin{proof}
Since every local blowing up can be decomposed as a finite sequence of simple local blowing ups (see Lemma \ref{Le_3}), we may assume that all local blowing ups in the sequence
\[
\left(\overline{R},\overline{\MI}\right)\lra \left(\overline{R}^{(1)},\overline{\MI}^{(1)}\right)\lra\cdots\lra \left(\overline{R}^{(r)},\overline{\MI}^{(r)}\right)
\]
are simple. We will prove by induction on $k,\ 1\leq k\leq r$ that we can lift the simple local blowing up
\[
\overline{\pi}_k:\overline{R}^{(k-1)}\lra \overline{R}^{(k)}
\]
$\left(\overline{R}^{(0)}:=R/\PI\right)$ to a simple local blowing up
\[
\pi_k:R^{(k-1)}\lra R^{(k)}
\]
$\left(R^{(0)}:=R\right)$ with respect to $\nu$ such that $R^{(k)}_{\PI^{(k)}}=R_\PI$ and $\overline{R}^{(k)}\cong R^{(k)}/\PI^{(k)}$.
For $k=1$ we apply Lemma \ref{Le_4} with $R=R$ and $\overline{R}=R/\PI$. Suppose now that $k>1$ and that $R^{(k-1)}_{\PI^{(k-1)}}=R_\PI$ and $\overline{R}^{(k-1)}\cong R^{(k-1)}/\PI^{(k-1)}$. Applying Lemma \ref{Le_4} to $R=R^{(k-1)}$ and $\overline{R}=\overline{R}^{(k-1)}$, we get $R^{(k)}_{\PI^{(k)}}=R^{(k-1)}_{\PI^{(k-1)}}=R_\PI$ and $\overline{R}^{(k)}\cong R^{(k)}/\PI^{(k)}$. Therefore, $R^{(r)}_{\PI^{(r)}}=R_\PI$ and $R^{(r)}/\PI^{(r)}\cong \overline{R}^{(r)}$, as desired.
\end{proof}

We will now assume that both $R_\PI$ and $R/\PI$ are regular and will study the effects of blowing up $R$ with respect to $\nu$.
\begin{Lema}\label{Le_5}
Let $R$ be a domain and $\PI$ a prime ideal of $R$ such that $R_\PI$ is regular. Then there exist $y_1,\ldots,y_r\in\PI$ that form a regular system of parameters for $R_\PI$.
\end{Lema}
\begin{proof}
Since $R_\PI$ is regular, there exist $\widetilde{y}_1,\ldots,\widetilde{y}_r\in \PI R_\PI$ which form a regular system of parameters for $R_\PI$. By definition of $\PI R_\PI$, there exist $\beta_i\notin\PI$ and $y_i\in \PI$ such that
\[
\widetilde{y}_i=\frac{y_i}{\beta_i},\qquad 1\leq i\leq r.
\]
Then $\beta_i$ is a unit in $R_\PI$ and therefore $\left(\widetilde{y}_1,\ldots,\widetilde{y}_r\right)R_\PI=(y_1,\ldots,y_r)R_\PI$.
\end{proof}

\begin{Lema}\label{Le_6}
Assume that $R_\PI$ and $R/\PI$ are regular, take $y_1,\ldots,y_r\in\PI$ which form a regular system of parameters for $R_\PI$ and $x_1,\ldots,x_t\in \MI\setminus\PI$ such that $(x_1+\PI,\ldots,x_t+\PI)$ form a regular system of parameters for $R/\PI$. Fix $a\in\MI\setminus\PI$ and let
\[
\pi:R\lra R^{(1)}
\]
be the local blowing up of $R$ with respect to $\nu$ along the ideal $(a,y_1,\ldots,y_r,y_{r+1},\ldots,y_{r+s})$ for some $y_{r+1},\ldots,y_{r+s}\in \PI$. Let
$$
y_1^{(1)}=\displaystyle\frac{y_1}{a},\ldots,y_r^{(1)}=\displaystyle\frac{y_r}{a}
$$
and $\PI^{(1)}=\MI_{\nu_1}\cap R^{(1)}$. Then we have:
\begin{description}
\item[(i)] $R^{(1)}/\PI^{(1)}$ is regular and $\left(x_1+\PI^{(1)},\ldots,x_t+\PI^{(1)}\right)$ is a regular system of parameters for $R^{(1)}/\PI^{(1)}$ and
\item[(ii)] $R^{(1)}_{\PI^{(1)}}$ is regular and $\left(y_1^{(1)},\ldots,y_r^{(1)}\right)$ is a regular system of parameters for $R^{(1)}_{\PI^{(1)}}$.
\end{description}
\end{Lema}

\begin{proof}
\begin{description}
\item[(i)] We want to prove first that $R^{(1)}/\PI^{(1)}\cong R/\PI$. Let $i:R\lra R^{(1)}$ be the natural inclusion and $\pi:R^{(1)}\lra R^{(1)}/\PI^{(1)}$ be the canonical epimorphism. Consider the mapping $\phi=\pi\circ i$. We have to prove that $\phi$ is surjective and that $\ker \phi=\PI$. Take an element $\displaystyle\frac{p}{q}\in R^{(1)}$, so $p,q\in R\left[\displaystyle\frac{y_1}{a},\ldots,\frac{y_{r+s}}{a}\right]$ and $\nu(q)=0$. Write
\[
p=p_0+p_1\textnormal{ and } q=q_0+q_1
\]
where $p_0,q_0\in R$ and
\[
p_1=\frac{y_1}{a}\overline{p}_1+\ldots+\frac{y_{r+s}}{a}\overline{p}_{r+s} \textnormal{ and } q_1=\frac{y_1}{a}\overline{q}_1+\ldots+\frac{y_{r+s}}{a}\overline{q}_{r+s}
\]
for some $\overline{p}_i,\overline{q}_i\in R\left[\displaystyle\frac{y_1}{a},\ldots,\frac{y_{r+s}}{a}\right]$. Since $\nu_1\displaystyle\left(\frac{y_i}{a}\right)=\nu_1(y_i)-\nu_1(a)>0$ we have $\nu_1\left(p_1\right)>0$ and $\nu_1\left(q_1\right)>0$, in particular $\nu\left(p_1\right)>0$ and $\nu\left(q_1\right)>0$. Since $\nu(q_1)>0$ and $\nu(q)=0$ we have
\[
\nu(q_0)=\nu(q+q_0-q)=\nu(q-q_1)=0
\]
and since $\nu$ is centered at $(R,\MI)$ we have $\displaystyle\frac{1}{q_0}\in R$, and consequently $\displaystyle\frac{p_0}{q_0}\in R$. Also,
\[
\frac{p}{q}=\frac{p_0+p_1}{q_0+q_1}=\frac{p_0}{q_0}+\frac{p_1q_0-p_0q_1}{q_0q}=:\frac{p_0}{q_0}+p^{(1)}.
\]
Since $\nu(q)=0=\nu(q_0)$ we have $\nu(q_0q)=0$ so $p^{(1)}\in R^{(1)}$. On the other hand,
\[
\nu_1(p_1q_0-p_0q_1)\geq \min\{\nu_1(p_1q_0),\nu_1(p_0q_1)\}>0
\]
and we have $r^{(1)}\in \PI^{(1)}$. Therefore,
\[
\frac{p}{q}+\PI^{(1)}=\frac{p_0}{q_0}+\PI^{(1)}=\phi\left(\frac{p_0}{q_0}\right)
\]
so $\phi$ is surjective. The fact that $\ker\phi=\PI$ is trivial because $\PI=R\cap \PI^{(1)}$. Therefore, $R^{(1)}/\PI^{(1)}\cong R/\PI$.

Since $R^{(1)}/\PI^{(1)}\cong R/\PI$ then $R^{(1)}/\PI^{(1)}$ is regular and the images
$$
x_1+\PI^{(1)},\ldots,x_t+\PI^{(1)}\in R^{(1)}/\PI^{(1)}
$$
of $x_1+\PI,\ldots,x_t+\PI\in R/\PI$ under $\phi$ form a regular system of parameters for $\MI^{(1)}/\PI^{(1)}$.

\item[(ii)]
Proving that $R_\PI=R^{(1)}_{\PI^{(1)}}$ is similar to the proof of Lemma \ref{Le_4}. Namely, the inclusion $R_\PI\subset R^{(1)}_{\PI^{(1)}}$ is obvious and the opposite inclusion follows from Lemma \ref{localization} (1), since $\displaystyle\frac{y_i}a\in R_\PI$.
 
 Since $\nu_1(a)=0$ the element $a$ is a unit in $R^{(1)}_{\PI^{(1)}}$. From this and from the fact that $(y_1,\ldots,y_r)$ is a regular system of parameters for $R_\PI$ we conclude that $\left(y_1^{(1)},\ldots,y_r^{(1)}\right)=\left(\displaystyle\frac{y_1}{a},\ldots, \frac{y_r}{a}\right)$ is a regular system of parameters for $R^{(1)}_{\PI^{(1)}}$.

\end{description}
\end{proof}

\subsection{The local uniformization problem}

We will now give the different definitions of local uniformization for a valuation centered in some Noetherian local ring.

\begin{Def}[\textbf{Local Uniformization Property}]
Let $(R,\MI)$ be a Noetherian local ring and $\nu$ a valuation of $K=Quot(R)$ centered at $\left(R,\MI\right)$. We say that $\nu$ admits \textbf{local uniformization} (or that $\nu$ has the \textbf{local uniformization property}) if there exists a sequence of local blowing ups
\[
\left(R,\MI\right)\lra\left(R^{(1)},\MI^{(1)}\right)\lra\cdots\lra\left(R^{(n)},\MI^{(n)}\right)
\]
with respect to $\nu$ such that $\left(R^{(n)},\MI^{(n)}\right)$ is regular.
\end{Def}

\begin{Def}[\textbf{Weak Embedded Local Uniformization Property}]
Let $\left(R,\MI\right)$ be a Noetherian local ring with quotient field $K$ and take a valuation $\nu$ of $K$ centered at $\left(R,\MI\right)$. We say that $\nu$ admits \textbf{weak embedded local uniformization} if for every given $f_1,\ldots,f_q\in R$ there exists a sequence of local blowing ups
\[
\left(R,\MI\right)\lra\left(R^{(1)},\MI^{(1)}\right)\lra\cdots \lra\left(R^{(n)},\MI^{(n)}\right)
\]
with respect to $\nu$ such that $\left(R^{(n)},\MI^{(n)}\right)$ is regular and there exists a regular system of parameters $u=\left(u_1,\ldots,u_d\right)$ of $R^{(n)}$ such that $f_i$ are monomials in $u$ for $1\leq i\leq q$.
\end{Def}

\begin{Def}[\textbf{Embedded Local Uniformization Property}]
Let $\left(R,\MI\right)$ be a Noetherian local ring with quotient field $K$ and take a valuation $\nu$ of $K$ centered at $\left(R,\MI\right)$. We say that $\nu$ admits \textbf{embedded local uniformization} if for every given finite set $f_1,\ldots,f_q\in R$ with $\nu\left(f_1\right)\leq \nu\left(f_2\right)\leq\ldots\leq \nu\left(f_q\right)$ there exists a sequence of local blowing ups
\[
\left(R,\MI\right)\lra\left(R^{(1)},\MI^{(1)}\right)\lra\cdots \lra\left(R^{(n)},\MI^{(n)}\right)
\]
with respect to $\nu$ such that $\left(R^{(n)},\MI^{(n)}\right)$ is regular and there is a regular system of parameters $u=\left(u_1,\ldots,u_d\right)$ of $R^{(n)}$ such that all $f_i$'s are monomials in $u$ and $f_1\mid f_2\mid \ldots\mid f_q$.
\end{Def}

\section{Proofs of the main results}
We will now prove the main results of this paper.

\subsection{Proof of Theorem \ref{Teo_1}}

We will proceed by induction on the rank of the valuation. Let $n$ be a given natural number and assume that every valuation $\mu$ centered at a Noetherian local ring
$$
\left(R',\MI'\right)\in\mathcal{M}
$$
with $rk\left(\mu\right)<n$ admits local uniformization. Take a valuation $\nu$ centered at a Noetherian local ring $\left(R,\MI\right)\in\mathcal{M}$ such that $rk(\nu)=n$. We will  prove that $\nu$ admits local uniformization.

Write $\nu=\nu_1\circ\nu_2$ with $rk\left(\nu_1\right)<rk\left(\nu\right)$ and $rk\left(\nu_2\right)<rk\left(\nu\right)$. Then $\nu_1$ is a valuation of $K$ with center $\mathfrak{p}\subseteq \MI$ in $R$ (so $\nu_1$ is centered at $R_\PI$) and $\nu_2$ is a valuation of $K_{\nu_1}$ whose restriction to $\kappa\left(\mathfrak p\right)$ is centered at $\left( R/\mathfrak p,\mathfrak m/\mathfrak p\right)$ (see Lemma \ref{Le_2} above).

Since $rk\left(\nu_1\right)<rk\left(\nu\right)$, by the induction assumption, there exists a sequence of local blowing ups
\[
R_\PI\lra \widetilde{R}^{(1)}\lra\cdots\lra \widetilde{R}^{(r)}
\]
with respect to $\nu_1$ such that $\widetilde{R}^{(r)}$ is regular. From Corollary \ref{Cor_2} we conclude that there exists a sequence of local blowing ups
\[
R\lra R^{(1)}\lra\cdots\lra R^{(n)}
\]
with respect to $\nu$ such that $R^{(n)}_{\PI^{(n)}}$ is regular, where $\PI^{(n)}$ is the center $\MI_{\nu_1}\cap R^{(n)}$ of $\nu_1$ on $R^{(n)}$. Replacing $R^{(n)}$ by $R$, we may assume that $R_\mathfrak{p}$ is regular.

Next, we apply the induction assumption to $R/\mathfrak{p}$. Since $rk(\nu_2)<rk(\nu)=n$ there exists a sequence of local blowing ups
\[
\left(R/\PI,\MI/\PI\right)\lra \left(\overline{R}^{(1)},\overline{\MI}^{(1)}\right)\lra\cdots\lra \left(\overline{R}^{(m)},\overline{\MI}^{(m)}\right)
\]
with respect to $\nu_2$ such that $\left(\overline{R}^{(m)},\overline{\MI}^{(m)}\right)$ is regular. By Corollary \ref{Cor_1} there exists a sequence of local blowing ups
\[
(R,\MI)\lra \left(R^{(1)},\MI^{(1)}\right)\lra\cdots\lra \left(R^{(m)},\MI^{(m)}\right)
\]
with respect to $\nu$ such that $R^{(m)}_{\PI^{(m)}}$ and $R^{(m)}/{\PI^{(m)}}$ are regular. Replacing $R$ by $R^{(m)}$, we may assume that both $R_\mathfrak{p}$ and $R/\mathfrak{p}$ are regular.

Let $(y_1,\ldots,y_r)\subseteq\PI$ be a regular system of parameters for $\mathfrak{p}R_\PI$ (we can take $y_i\in \PI$ by Lemma \ref{Le_5}), and $x_1,\ldots, x_t$ a set of elements of $R\setminus\mathfrak{p}$, whose images modulo $\mathfrak{p}$ form a regular system of parameters of $\mathfrak{m}/\mathfrak{p}$. If $y_1,\ldots,y_r$ generate $\PI$ then $R$ is regular. Indeed, since $y_1,\ldots,y_r,x_1,\ldots,x_t$ generate $\MI$ we have $r+t\geq \dim R$. Also, since $r= \dim R_\mathfrak{p}=ht\left(\PI\right)$ and $t=\dim \left(R/\PI\right)=ht\left(\MI/\PI\right)$ we have
\[
\dim R=ht\left(\MI\right)\geq ht\left(\PI\right)+ht\left(\MI/\PI\right)=r+t\geq \dim R.
\]
Therefore, $r+t=\dim R$ and $x_1,\ldots,x_t,y_1,\ldots,y_s$ is a minimal set of generators of $\MI$, hence $\left(R,\MI\right)$ is regular.

If $y_1,\ldots,y_r$ do not generate $\PI$, take $y_{r+1},\ldots,y_{r+s}\in \PI$ such that $y_1,\ldots,y_r,y_{r+1},\ldots,y_{r+s}$ generate $\PI$. Since the residues of $y_1,\ldots, y_r$ modulo $\left(\PI R_\PI\right)^2$ form a $\kappa\left(\PI\right)$-basis of $\PI R_\PI/\left(\PI R_\PI\right)^2$, for each $k=1,\ldots, s$ we can find an equation
\[
a_{k}y_{r+k}+b_{1k}y_1+\ldots +b_{rk}y_r-h_k=0
\]
where $a_k\in R\setminus \PI$ and $h_k\in \left(\PI R_\PI\right)^2$. In fact, multiplying the above equations by suitable elements of $R\setminus\PI$, we may assume that
\begin{equation}
h_k\in \left(y_1,\ldots ,y_r\right)^2,\qquad 1\leq k\leq s.\label{eq:hinysquare}
\end{equation}
First, let us blow up $R$ with respect to $\nu$ along the ideal $\left(a_1,y_1,\ldots,y_r\right)$ obtaining a new local ring $\left(R^{(1)},\MI^{(1)}\right)$. In $R^{(1)}$ we have $y_1=a_1y_1^{\left(1\right)}, y_2=a_1y_2^{\left(1\right)},\ldots, y_r=a_1y_r^{\left(1\right)}$ and we rewrite the previous relations as
\[
a_{k}y_{r+k}+a_1b_{1k}y^{(1)}_1+\ldots +a_1b_{rk}y^{(1)}_r-h_k=0,\qquad 1\leq k\leq s
\]
Observe that by (\ref{eq:hinysquare}) we have $h_k\in a_1^2\left(y_1^{\left(1\right)},\ldots,y_r^{\left(1\right)}\right)^2$. In particular, we have $a_1^2\mid h_1$ in $R^{(1)}$ and we obtain
\[
a_1\left(y_{r+1}^{\left(1\right)}+b_{11}y_1^{\left(1\right)}+\ldots+b_{r1}y_r^{\left(1\right)}-h_1'\right)=0
\]
and
\[
a_{k}y_{r+k}+a_1b_{1k}y^{(1)}_1+\ldots +a_1b_{rk}y^{(1)}_r-h_k=0
\]
for $k>1$, where $h_1=a_1h_1'$ with $h_1',h_2,\ldots,h_s\in \left(y_1^{(1)},\ldots,y_r^{(1)}\right)^2$. In particular,
\begin{equation}\label{2}
y_{r+1}^{\left(1\right)}+b_{11}y_1^{\left(1\right)}+\ldots+b_{r1}y_r^{\left(1\right)}-h_1'=0
\end{equation}
Since $h'_1\in\left(y_1^{(1)},\ldots,y_r^{(1)}\right)$ we have $y_{r+1}\in \left(y_1^{(1)},\ldots,y_r^{(1)}\right)$ and consequently
\[
\PI^{(1)}=\left(y_{r+2},\ldots,y_{r+s},y_1^{(1)},\ldots,y_r^{(1)}\right).
\]
By Lemma \ref{Le_6}, $\left(x_1+\PI^{(1)},\ldots,x_t+\PI^{(1)}\right)$ is a regular system of parameters for $R^{(1)}/\PI^{(1)}$ and $\left(y^{(1)}_1,\ldots,y^{(1)}_r\right)$ is a regular system of parameters for $R^{(1)}_{\PI^{(1)}}$.

We proceed as before with $a_k$ for all $k=2,\ldots,s$ until we reach a local ring $R^{(s)}$ for which $\PI^{(s)}=\left(y_1^{(s)},\ldots,y_r^{(s)}\right)R^{(s)}$, $\left(x_1+\PI^{(s)},\ldots,x_t+\PI^{(s)}\right)$ is a regular system of parameters for $R^{(s)}/\PI^{(s)}$ and $\left(y_1^{(s)},\ldots,y_r^{(s)}\right)$ is a regular system of parameters for $R^{(s)}_{\PI^{(s)}}$. Therefore, $R^{(s)}$ is regular with regular system of parameters $\left(x_1,\ldots,x_t,y^{(s)}_1,\ldots,y^{(s)}_r\right)$.

\subsection{Proof of Theorem \ref{Teo_2}}

We will proceed as before. Let $\nu$ be a valuation centered at $\left(R,\MI\right)$ with $rk\left(\nu\right)>1$, decompose it as $\nu=\nu_1\circ \nu_2$ and let $\PI$ be the center of $\nu_1$ on $R$. We want to prove that given $f_1,\ldots,f_q\in R$, there exists a sequence of local blowing ups
\[
\left(R,\MI\right)\lra \left(R^{\left(1\right)},\MI^{\left(1\right)}\right)\lra\cdots\lra\left(R^{(m)},\MI^{(m)}\right)
\]
with respect to $\nu$ such that $\left(R^{\left(m\right)},\MI^{\left(m\right)}\right)$ is regular and there exists a regular system of parameters $u^{\left(m\right)}=\left(u_1^{\left(m\right)},\ldots,u_d^{\left(m\right)}\right)$ of $\MI^{\left(m\right)}$ such that $f_1,\ldots, f_q$ are monomials in $u^{\left(m\right)}$.

By the induction hypothesis, there exists a sequence of local blowing ups
\[
R_\PI\lra \widetilde{R}^{(1)}\lra\cdots\lra \widetilde{R}^{(m)}
\]
with respect to $\nu_1$ such that $\widetilde{R}^{(m)}$ is regular and there exists a regular system of parameters $z=\left(z_1,\ldots,z_r\right)$ of $\widetilde{R}^{\left(m\right)}$ such that $f_i=c_i z^{\gamma_i}$ where $c_i$ is a unit in $\widetilde{R}^{\left(m\right)}$. From Corollary \ref{Cor_2} we conclude that there exists a sequence of local blowing ups
\[
R\lra R^{(1)}\lra\cdots\lra R^{(n)}
\]
with respect to $\nu$ such that $R^{(n)}_{\PI^{(n)}}=\widetilde{R}^{(m)}$, where $\PI^{(n)}$ is the center $\MI_{\nu_1}\cap R^{(n)}$ of $\nu_1$ on $R^{(n)}$. Replacing $R^{(n)}$ by $R$ we may assume that $R_\PI$ is regular with regular system of parameters $z$ such that $f_i=c_i z^{\gamma_i}$ with $c_i$ a unit in $R_\PI$. Writing $c_i=\displaystyle\frac{\alpha_i}{\beta_i}$ with $\alpha_i,\beta_i\in R\setminus \PI$ we get
\[
f_i=\frac{\alpha_i}{\beta_i} z^{\gamma_i},\qquad 1\leq i\leq q.
\]
Moreover, we may assume that $z_j\in \PI$. Indeed, since $z_j\in \PI R_{\PI}$ we can write $z_j=\displaystyle\frac{1}{a_j}y_j$ with $y_j\in \PI$ and $a_j\in R\setminus \PI$. We have $\PI R_\PI=\left(z_1,\ldots,z_r\right)R_{\PI}=\left(y_1,\ldots,y_r\right)R_{\PI}$ and defining $\displaystyle\beta_i'=\beta_i\prod_{j=1}^r a_j^{\gamma_i^{\left(j\right)}}$ we have
\[
f_i=\frac{\alpha_i}{\beta'_i} y^{\gamma_i},\qquad 1\leq i\leq q.
\]
Blowing up $R$ with respect to $\nu$ along the ideals $\left(\beta_i',y_1,\ldots,y_r\right)$, we may assume that $\beta_i'=1$.

From the previous paragraph we can assume that $R_\PI$ is regular and that there are $y_1,\ldots, y_r\in\PI$ that form a regular system of parameters for $R_\PI$ and there exist $\alpha_i\in R\setminus \PI$ such that
\[
f_i=\alpha_i y^{\gamma_i},\qquad 1\leq i\leq q.
\]

Extend now $\left(y_1,\ldots,y_r\right)$ to a set of generators of $\PI$, say $\left(y_1,\ldots,y_r,y_{r+1},\ldots, y_{r+s}\right)$. Since the residues of $y_1,\ldots, y_r$ modulo $\left(\PI R_\PI\right)^2$ form a $\kappa\left(\PI\right)$-basis of $\PI R_\PI/\left(\PI R_\PI\right)^2$, for each $k=1,\ldots, s$ we can find an equation
\[
a_{k}y_{r+k}+b_{1k}y_1+\ldots +b_{rk}y_r-h_k=0
\]
where $a_k\in R\setminus \PI$ and $h_k\in \left(\PI R_\PI\right)^2$. Multiplying the last equation by a suitable element of $R\setminus \PI$ we can assume that $h_k\in \left(\PI\right)^2$.

By the induction assumption for $\nu_2$, there exists a sequence of local blowing ups
\[
R/\PI\lra \overline{R}^{(1)}\lra\cdots\lra\overline{R}^{(r)}
\]
such that $\overline{R}^{(r)}$ is regular and there exists a regular system of parameters $\overline{x}=\left(\overline{x}_1,\ldots,\overline{x}_t\right)$ of $\overline{R}^{(r)}$ such that the residues of $\alpha_i$ and $a_k$ (modulo $\PI$) are monomials in $\overline{x}$, i.e.,
\[
\overline{\alpha}_i= \overline{u}_i\overline{x}^{\delta_i},\qquad 1\leq i\leq q
\]
and
\[
\overline{a}_k= \overline{v}_k\overline{x}^{\epsilon_k},\qquad 1\leq k\leq s
\]	
where $\overline{u}_i,\overline{v}_k$ are units in $\overline{R}$. By Corollary \ref{Cor_1} there exists a sequence of local blowing ups
\[
R\lra R^{(1)}\lra\cdots\lra R^{(n)}
\]
with respect to $\nu$ such that $R^{(n)}_{\PI^{(n)}}=R_\PI$ and $R^{(n)}/\PI^{(n)}\cong \overline{R}^{(r)}$ (for $x\in R^{(n)}\setminus\PI^{(n)}$ denote by $\overline{x}$ the element corresponding to $x+\PI^{(n)}$ via this isomorphism and say that $x$ represents $\overline{x}$). Choose elements $x_l,u_i, v_k\in R^{(n)}\setminus \PI^{(n)}$ that represent $\overline{x}_l,\overline{u}_i, \overline{v}_k$ respectively. Then
\[
\alpha_i=u_ix^{\delta_i}+r_i,\qquad 1\leq i\leq q
\]
and
\[
a_k=v_kx^{\epsilon_k}+s_k,\qquad 1\leq k\leq s
\]
for some $r_i,s_k\in \PI^{(n)}$.

From the last paragraphs we may assume that $R_\PI$ is regular with a regular system of parameters $y=\left(y_1,\ldots,y_r\right)$ which extends to a set of generators $\left(y_1,\ldots, y_r,y_{r+1},\ldots,y_{r+s}\right)$ of $\PI$ and there exist $x_1,\ldots,x_t\in \MI\setminus\PI$ such that their images in $R/\PI$ form a regular system of parameters of $R/\PI$ such that
\begin{equation}\label{1}
f_i=\left(u_i x^{\delta_i}+r_i\right) y^{\gamma_i},\qquad 1\leq i\leq q
\end{equation}
and
\begin{equation}\label{4}
v_k x^{\epsilon_k} y_{r+k}+b_{1k}y_1+\ldots+b_{rk}y_r+h'_k=0,\qquad 1\leq k\leq s,
\end{equation}
where $u_i,v_k$ are units in $R$, and $r_i, s_k\in\PI$ and $h'_k=h_k+s_k y_{r+k}\in \left(\PI\right)^2$.

From now on we will just blow up $R$ with respect to $\nu$ along ideals of the form $\left(x_l,y_1,\ldots, y_r\right)$ or $\left(x_l,y_1,\ldots, y_r,y_{r+s_1},\ldots,y_{r+s}\right)$ for some $1\leq s_1\leq s$. Take $l\in {1,\ldots,t}$ such that $x_l\mid x^{\delta_i}$ for some $i=1,\ldots,q$. Blowing up $R$ with respect to $\nu$ along
\[
\left(x_l,y_1,\ldots,y_r,y_{r+1},\ldots,y_{r+s}\right)
\]
we obtain a system of generators
\[
\left(x_1,\ldots,x_t,y^{\left(1\right)}_1,\ldots,y^{\left(1\right)}_r,y^{\left(1\right)}_{r+1},\ldots,y^{\left(1\right)}_{r+s}\right)
\]
of $\MI^{(1)}$ such that $y_j=x_l y^{\left(1\right)}_j$ for all $j=1,\ldots,r+s$. Substituting this new system of generators to the equations (\ref{1}) and (\ref{4}), we obtain
\begin{equation}\label{5}
f_i=\left(u_i x^{\delta_i}+x_l r^{(1)}_i\right) x_l^{|\gamma_i|}\left(y^{\left(1\right)}\right)^{\gamma_i}=\left(u_i \frac{x^{\delta_i}}{x_l}+r^{(1)}_i\right) x_l^{|\gamma_i|+1}\left(y^{\left(1\right)}\right)^{\gamma_i},\qquad 1\leq i\leq q
\end{equation}
and
\begin{equation}\label{6}
v_k x^{\epsilon^{\left(1\right)}_k} y^{\left(1\right)}_{r+k}+b_{1k}^{(1)}y^{\left(1\right)}_1+\ldots+b_{rk}^{(1)}y^{\left(1\right)}_r+h'_k=0,\qquad 1\leq k\leq s,
\end{equation}
where $r^{(1)}_i\in \PI^{\left(1\right)}$. Observe that $h'_k\in \left(y_1^{(1)},\ldots, y_r^{(1)},y^{\left(1\right)}_{r+1},\ldots,y^{\left(1\right)}_{r+k}\right)^2$. After finitely many of these local blowing ups we get a local ring $\left(R^{\left(n\right)},\MI^{\left(n\right)}\right)$ such that $\MI^{\left(n\right)}$ is generated by
\[
\left(x_1,\ldots,x_t,y^{\left(n\right)}_1,\ldots, y^{\left(n\right)}_r,y^{\left(n\right)}_{r+1},\ldots,y^{\left(n\right)}_{r+s}\right)
\]
with
\[
f_i=\left(u_i+r^{(n)}_i\right) x^{\tau_i}  \left(y^{\left(n\right)}\right)^{\gamma_i}=u_i' x^{\tau_i}  \left(y^{\left(n\right)}\right)^{\gamma_i},\qquad 1\leq i\leq q
\]
and
\begin{equation}\label{Eq_11}
v_k x^{\epsilon^{\left(n\right)}_k} y^{\left(n\right)}_{r+k}+b_{1k}^{(n)}y^{\left(n\right)}_1+\ldots+b_{rk}^{(n)}y^{\left(n\right)}_r+h'_k=0,\qquad 1\leq k\leq s
\end{equation}
with $h'_k\in \left(y_1^{(n)},\ldots, y_r^{(n)},y^{\left(n\right)}_{r+1},\ldots,y^{\left(n\right)}_{r+k}\right)^2$ and $u_i'\in \left(R^{(n)}\right)^\times$. Therefore, all $f_i$ are monomials in
\[
\left(x,y^{(n)}\right):=\left(x_1,\ldots,x_t,y^{\left(n\right)}_1,\ldots, y^{\left(n\right)}_r\right).
\]
Observe that if we blow up $R^{(n)}$ with respect to $\nu$ along ideals of the form
\begin{equation}\label{Eq_10}
\left(x_l,y^{\left(n\right)}_1,\ldots, y^{\left(n\right)}_r\right)\textnormal{ or }\left(x_l,y^{\left(n\right)}_1,\ldots, y^{\left(n\right)}_r,y_{r+s_1},\ldots, y_{r+s}\right)
\end{equation}
for some $s_1\in\left\{1,\ldots,s\right\}$ then all $f_i's$ will be monomials in $\left(x,y^{(n+1)}\right)$.

We still do not have that $R^{(n)}$ is regular. In order to obtain that, we will blow up $R^{(n)}$ with respect to $\nu$ along ideals of the form (\ref{Eq_10}). Let $x_l\mid x^{\epsilon^{\left(n\right)}_1}$ for some $1\leq l\leq t$ and blow up $R^{(n)}$ with respect to $\nu$ along the ideal
\[
\left(x_l,y_1^{(n)},\ldots,y_r^{(n)}\right).
\]
In $R^{(n+1)}$ equation (\ref{Eq_11}) for $k=1$ can be rewritten as
\begin{equation}\label{Eq_12}
v_1 x^{\epsilon^{\left(n\right)}_1} y^{\left(n+1\right)}_{r+1}+x_l\left(b_{11}^{(n)}y^{\left(n+1\right)}_1+\ldots+b_{r1}^{(n)}y^{\left(n+1\right)}_r\right)+h'_1=0.
\end{equation}
with $h'_1\in \left(y_1^{(n+1)},\ldots, y_r^{(n+1)},y^{\left(n+1\right)}_{r+1},\ldots,y^{\left(n+1\right)}_{r+k}\right)^2$. Now we blow up $R^{(n+1)}$ with respect to $\nu$ along
\[
\left(x_l,y_1^{(n+1)},\ldots,y_r^{(n+1)},y^{\left(n+1\right)}_{r+1},\ldots,y^{\left(n+1\right)}_{r+s}\right)
\]
and the equation (\ref{Eq_11}) rereads as
\[
x_l^2\left(v_1\frac{x^{\epsilon^{(n)}_1}}{x_l} y^{\left(n+2\right)}_{r+1}+b_{11}^{(n)}y^{\left(n+2\right)}_1+\ldots+b_{r1}^{(n)}y^{\left(n+2\right)}_r+h''_1\right)=0
\]
and consequently
\[
v_1 \frac{x^{\epsilon^{(n)}_1}}{x_l} y^{\left(n+2\right)}_{r+1}+b_{11}^{(n)}y^{\left(n+2\right)}_1+\ldots+b_{r1}^{(n)}y^{\left(n+2\right)}_r+h''_1=0,
\]
with $h''_1\in \left(y_1^{(n+2)},\ldots, y_r^{(n+2)},y_{r+1}^{(n+2)},\ldots,y_{r+k}^{(n+2)}\right)^2$. After finitely many of these steps we reach a ring $R^{(n+m_1)}$ where
\[
v_1 y^{\left(n+m_1\right)}_{r+1}+b_{11}^{(n)}y^{\left(n+m_1\right)}_1+\ldots+b_{r1}^{(n)}y^{\left(n+m_1\right)}_r+h^{(m_1)}_1=0,
\]
with $h^{(m_1)}_1\in \left(y_1^{(n+m_1)},\ldots, y_r^{(n+m_1)},y_{r+1}^{(n+m_1)},\ldots,y_{r+k}^{(n+m_1)}\right)^2$. It follows  now that
\[
\PI^{(n+m_1)}=\left(y_1^{(n+m_1)},\ldots, y_r^{(n+m_1)},y_{r+2}^{(n+m_1)},\ldots,y_{r+k}^{(n+m_1)}\right).
\]

Repeating the process as above for each $k=2,\ldots,s$, we reach a ring $R^{(n+m)}$ such that $\PI^{(n+m)}$ is generated by
\[
\left(y_1^{(n+m)},\ldots, y_r^{(n+m)}\right).
\]

Analogously to the proof of Theorem \ref{Teo_1}, we note that $\MI^{\left(n+m\right)}$ is generated by
\[
\left(x,y^{\left(n+m\right)}\right)=\left(x_1,\ldots,x_t,y^{\left(n+m\right)}_1,\ldots,y^{\left(n+m\right)}_r\right),
\]
so $R^{\left(n+m\right)}$ is regular with regular system of parameters $\left(x,y^{\left(n+m\right)}\right)$. Moreover, $f_i$ is a monomial on $\left(x,y^{\left(n+m\right)}\right)$ for each $i=1,\ldots,q$. Therefore, we have achieved weak embedded local uniformization for $\nu$.

\subsection{Proof of Theorem \ref{Teo_3}}

Let $\nu$ be a valuation with $rk\left(\nu\right)>1$. We want to prove that given $f_1,\ldots,f_q\in R$ such that $\nu\left(f_1\right)\leq\ldots\leq\nu\left(f_q\right)$ there exists a sequence of local blowing ups
\[
\left(R,\MI\right)\lra \left(R^{(1)},\MI^{(1)}\right)\lra\cdots\lra \left(R^{(n)},\MI^{(n)}\right)
\]
with respect to $\nu$ such that $\left(R^{(n)},\MI^{(n)}\right)$ is regular and there exists a regular system of parameters $u^{(n)}=\left(u^{(n)}_1,\ldots,u^{(n)}_d\right)$ of $\left(R^{(n)},\MI^{(n)}\right)$ such that $f_i$ is a monomial in $u^{(n)}$ for all $i=1,\ldots,q$ and $f_1\mid_{R^{(n)}}\ldots\mid_{R^{(n)}} f_q$.

Again, we will proceed by induction on the rank. Write $\nu=\nu_1\circ\nu_2$ with $rk\left(\nu_2\right)=1$. By induction hypothesis for $\nu_1$ and after changes as in Propositions \ref{Teo_1} and \ref{Teo_2} we can assume that $R_\PI$ is regular and there exists $y_1,\ldots,y_r\in\PI$ that form a regular system of parameters for $R_\PI$ such that
\[
f_i=\alpha_i y^{\gamma_i},\qquad 1\leq i\leq q
\]
with $\alpha_i\in R\setminus\PI$ and $y^{\gamma_1}\mid_R\ldots\mid_Ry^{\gamma_q}$.

We want to modify $\alpha_i$ in such a way that $\nu_2\left(\alpha_1+\PI\right)\leq\ldots\leq \nu_2\left(\alpha_q+\PI\right)$. We will do that by blowing up $R$ with respect to $\nu$ along an ideal of the form $\left(x^n,y_1,\ldots,y_r\right)$ for some $x\in R\setminus\PI$ and some $n\in \N$. Since $y^{\gamma_1}\mid_R\ldots\mid_R y^{\gamma_q}$ we have that $\gamma_1\leq \ldots\leq \gamma_r$ where $``\leq"$ is the componentwise partial order of $\left(\N\cup\{0\}\right)^r$. If $\gamma_i=\gamma_{i+1}$ for some $i=1,\ldots, q-1$ then $\nu\left(\alpha_i\right)\leq \nu\left(\alpha_{i+1}\right)$ so $\nu_2\left(\alpha_i+\PI\right)\leq\nu_2\left(\alpha_{i+1}+\PI\right)$. Fix $x\in R\setminus\PI$ such that $\nu_2\left(x+\PI\right)>0$. Since $rk\left(\nu_2\right)=1$, for each $i=1,\ldots,q-1$ such that $\gamma_i<\gamma_{i+1}$ we find $n_i\in \N$ such that
\[
\nu_2\left(\alpha_i+\PI\right)+n |\gamma_i| \nu_2\left(x+\PI\right)<\nu_2\left(\alpha_{i+1}+\PI\right)+n |\gamma_{i+1}| \nu_2\left(x+\PI\right)
\]
for all $n\geq n_i$. Choose $n\in \N$ such that $n\geq n_i$ for all $i\in\{1,\ldots,q-1\}$ with $|\gamma_i|<|\gamma_{i+1}|$. Blowing up $R$ with respect to $\nu$ along the ideal $\left(x^n,y_1,\ldots,y_r\right)$ we obtain
\[
f_i=\alpha_i' \left(y^{(1)}\right)^{\gamma_i}=\alpha_i x^{n|\gamma_i|} \left(y^{(1)}\right)^{\gamma_i}, \qquad 1\leq i\leq q.
\]
Thus,
\[
\nu_2\left(\alpha'_i+\PI^{(1)}\right)=\nu_2\left(\alpha_ix^{n|\gamma_i|}+\PI^{(1)} \right)=\nu_2\left(\alpha_i+\PI^{(1)}\right)+n |\gamma_i| \nu_2\left(x+\PI^{(1)}\right)
\]
and consequently
\[
\nu_2\left(\alpha'_i+\PI^{(1)}\right)\leq \nu_2\left(\alpha'_{i+1}+\PI^{(1)}\right) ,\qquad  1\leq i\leq q-1.
\]

From the last paragraphs, we can assume that $R_\PI$ is regular and there exist $y_1,\ldots,y_r\in \PI$ that form a regular system of parameters for $R_\PI$ such that
\begin{equation}\label{7}
f_i=\alpha_i y^{\gamma_i},\qquad 1\leq i\leq q,
\end{equation}
where $y^{\gamma_1}\mid\ldots\mid y^{\gamma_q}$ and $\nu_2\left(\alpha_1+\PI\right)\leq\ldots\leq \nu_2\left(\alpha_q+\PI\right)$.

Extend $\left(y_1,\ldots,y_r\right)$ to a set of generators
\[
\left(y_1,\ldots,y_r,y_{r+1},\ldots,y_{r+s}\right)
\]
of $\PI$. As in the proofs of Theorem \ref{Teo_1} and \ref{Teo_2}, we have relations of the form
\begin{equation}\label{8}
a_{k}y_{r+k}+b_{1k}y_1+\ldots +b_{rk}y_r-h_k=0,\qquad 1\leq k\leq s,
\end{equation}
where $a_k\in R\setminus \PI$ and $h_k\in \PI^2$.

By the induction hypothesis for $R/\PI$ and after lifting the sequence of local blowing ups as in Propositions \ref{Teo_1} and \ref{Teo_2} we can assume that there exist $x_1,\ldots,x_t\in\MI\setminus\PI$ such that their images $x_1+\PI,\ldots,x_t+\PI$ form a regular system of parameters for $R/\PI$ and we have the following relations:
\[
\alpha_i=u_i x^{\epsilon_i}+r_i,\qquad 1\leq i\leq q,
\]
and
\[
a_k=v_k x^{\delta_k}+s_k, \qquad 1\leq k\leq s,
\]
where $u_i,v_k$ are units in $R$ and $s_k,r_i\in \PI$ for $k=1,\ldots,s$ and $i=1,\ldots, r$ with $x^{\epsilon_1}\mid\ldots\mid x^{\epsilon_q}$.

Substituting $a_k$'s and $\alpha_i$'s in equations (\ref{7}) and (\ref{8}) we obtain
\[
f_i=\left(u_i x^{\epsilon_i}+r_i\right) y^{\gamma_i},\qquad 1\leq i\leq q,
\]
and
\[
v_k x^{\delta_k} y_{r+k}+b_{1k}y_1+\ldots+b_{rk}y_r-h'_k=0,\qquad 1\leq k\leq s,
\]
where $h'_k\in \PI^2$.

Blowing up $R$ with respect to $\nu$ along ideals of the form $\left(x_l,y_1,\ldots,y_r,y_{r+1},\ldots,y_{r+s}\right)$ we have new coordinates $y^{(1)}=\left(y^{(1)}_1,\ldots,y^{(1)}_r\right)$ in which $y_j=x_l y^{(1)}_j$, $j=1,\ldots,r$. Therefore,
\[
f_i=\left(u_i x^{\epsilon_i}+x_l r^{(1)}_i\right) x_l^{|\gamma_i|} \left(y^{(1)}\right)^{\gamma_i},\qquad 1\leq i\leq q,
\]
where $r'_i\in \PI'$. If $x_l\mid x^{\epsilon_i}$, this equation can be rewritten as
\[
f_i=\left(u_i \frac{x^{\epsilon_i}}{x_l}+ r^{(1)}_i\right) x_l^{|\gamma_i|+1} \left(y^{(1)}\right)^{\gamma_i},\qquad 1\leq i\leq q.
\]
After finitely many of these local blowing ups we achieve
\[
f_i=\left(u_i+r^{(n)}_i\right) x^{\delta_i} \left(y^{(n)}\right)^{\gamma_i},\qquad  1\leq i\leq q
\]
where $\delta_i^{\left(l\right)}=\epsilon_i^{\left(q\right)} |\gamma_i|+\epsilon_i^{\left(l\right)}$. Since $\epsilon_1^{\left(l\right)}\leq\ldots\leq \epsilon_q^{\left(l\right)}$ we have that $x^{\delta_1}\mid\ldots\mid x^{\delta_q}$. Therefore, we achieved that $f_1,\ldots,f_q$ are monomials in $\left(x,y^{(n)}\right)$ and that $f_1\mid\ldots \mid f_q$.

We still don't have that $\left(R^{(n)},\MI^{(n)}\right)$ is regular. We can achieve that proceeding as in Theorem \ref{Teo_2}.
Now $\left(R^{(n+m)},\MI^{(n+m)}\right)$ is regular with regular system of parameters $\left(x,y^{(n+m)}\right)$ in which $f_i$ are monomials and $f_1\mid\ldots\mid f_q$. Therefore, we achieved embedded local uniformization for $\nu$.

\smallskip

\noindent
\bf Acknowledgement. \rm The first author would like to thank to his supervisor Franz-Viktor Kuhlmann. He was the one who posted the question and encouraged to write this paper. Also, his comments and suggestions were very helpful and clarifying.

\end{document}